\documentclass[12pt, reqno, draft]{amsart}

\usepackage{ifdraft}
\usepackage{fullpage}

\usepackage{mathtools}
\usepackage{graphicx}
\usepackage[dvipsnames]{xcolor}
\usepackage{pgf}

 \usepackage{tikz}
 \usetikzlibrary{matrix,arrows,decorations.pathmorphing,decorations.pathreplacing, arrows.meta}
 \usepackage{tikz-cd}
 \usetikzlibrary{arrows}

\usepackage{comment}

\usepackage[all]{xy}

\usepackage[alphabetic]{amsrefs}

\usepackage{amsmath}
\usepackage{mathrsfs}
\usepackage{epsfig}
\usepackage{amsfonts,amsthm,amscd, amssymb}
\usepackage{latexsym}
\usepackage{enumerate}
\usepackage{stmaryrd}
\usepackage{mypack}
\makeatother
\usepackage{hyperref}
\hypersetup{
    colorlinks=false,
    linkcolor=black,
   filecolor=black,
    urlcolor=black,
}

\makeatletter
\let\c@equation=\c@subsubsection

\makeatother

\definecolor{bluegray}{rgb}{0.4, 0.6, 0.8}

\title{{Commutative classifying space for simplicial groups}}
\date{\today}

\author{C\.{i}han Okay}
\email{cihan.okay@bilkent.edu.tr}
\address{Department of Mathematics, Bilkent University, 06800 Ankara}

\author{P\'al Zs\'amboki}
\email{zsamboki@renyi.hu}
\address{HUN-REN Alfr\'ed R\'enyi Institute of Mathematics, Re\'altanoda street 13-15, H-1053, Budapest, Hungary}

\theoremstyle{plain}

\newtheorem{innercustomthm}{Theorem}
\newenvironment{customthm}[1]
  {\renewcommand\theinnercustomthm{#1}\innercustomthm}
  {\endinnercustomthm}

\newtheorem{innercustomcor}{Corollary}
\newenvironment{customcor}[1]
  {\renewcommand\theinnercustomcor{#1}\innercustomcor}
  {\endinnercustomcor}

\definecolor{revC}{rgb}{0.0, 0.0, 1.0}
\newcommand\revC[1]{{\color{black}#1}}

\newcommand\revtwo[1]{{\color{black}#1}} %

\newcommand\revfour[1]{{\color{black}#1}} %

\newcommand\revfive[1]{{\color{black}#1}} %

\begin{document}

\begin{abstract} 
In this paper, we introduce a simplicial analog of classifying spaces for commutativity which classify principal bundles with commutativity structure on their transition functions.
Our construction $\Wbar(\tau,K)$, which takes as input a simplicial group $K$ and a cosimplicial group $\tau$ that encodes the additional structure such as commutativity, is a variation of the $\Wbar$-construction for simplicial groups. 
Our main result shows that the geometric realization of our $\Wbar(\tau,K)$ is homotopy equivalent to the topological classifying space $B(\tau,|K|)$.
\end{abstract} 

  \maketitle  
  
  \tableofcontents
  
\section{Introduction}

A classical result in algebraic topology is that principal $G$-bundles for a topological group $G$ are classified by the classifying space $BG$.
Principal $G$-bundles with a commutativity structure on their transition functions are introduced in \cite{AG15}. The classifying space $B(\ZZ,G)$ for such bundles, first introduced in \cite{ACT12}, is a variant of the ordinary classifying space $BG$, {that} is constructed from pairwise commuting group elements. 
This construction is a particular case of a class of constructions, denoted by $B(\tau,G)$, that depend on a cosimplicial group $\tau^\dt$.

\revfive{Our construction of the various classifying space models will be approached from a cosimplicial point of view.}
\revfive{Let $F^\bullet:\catDelta\to\catGrp$ denote the cosimplicial group defined as follows \cite[Definition 2.9]{Vil17}:
  \begin{enumerate}
    \item On the level of objects, it sends a nonnegative integer $n$ to the free group $F^n$ on generators $\{e_1,\dotsc,e_n\}$.
\item On the level of morphisms, we have, for $1\le j\le n$:
\begin{alignat*}{3}
d^0 e_j \; &=\; e_{j+1},
\qquad\qquad &
s^0 e_j \; &=\;
\begin{cases}
1 & j=1,\\[2pt]
e_{j-1} & j>1;
\end{cases}
\\[8pt]
d^i e_j \; &=\;
\begin{cases}
e_j & j<i,\\[2pt]
e_j e_{j+1} & j=i,\\[2pt]
e_{j+1} & j>i,
\end{cases}
\quad
\qquad\qquad &
s^i e_j \; &=\;
\begin{cases}
e_j & j\le i,\\[2pt]
1 & j=i+1,\\[2pt]
e_{j-1} & j>i+1,
\end{cases}
\qquad (0<i<n)
\\[8pt]
d^n e_j \; &=\; e_j,
\qquad\qquad &
s^n e_j \; &=\;
\begin{cases}
e_j & j\le n,\\[2pt]
1 & j=n+1.
\end{cases}
\end{alignat*}
  \end{enumerate}}
  Note that for any (non-simplicial) group $G$, we get a canonical isomorphism of simplicial sets $\Hom_{\catGrp}(F^\bullet, G)\cong\Wbar G$.

\noindent
\revfive{In constructing the variants of classifying spaces, we will work with quotients $\tau^\bullet$ of the cosimplicial group $F^\bullet$.} 
To be able to have better control on the resulting objects \revfour{$B(\tau,G)$}, {we assume the existence of a map} $\eta^\bullet:F^\bullet\to\tau^\bullet$ {of cosimplicial groups that is surjective in each degree}.
Then, the precomposition maps $\circ\eta^n:\Hom(\tau^n,G)\to\Hom(F^n,G)\cong G^{\times n}$ are injective. Thus, we can equip the set $\Hom(\tau^n,G)$ of group homomorphisms with the subspace topology, getting us the simplicial topological space $B(\tau,G)_\bullet=\Hom(\tau^\bullet,G)$. The space $B(\tau,G)$ is the geometric realization $|B(\tau,G)_\bullet|$.
The main examples are $BG$, the usual classifying space, corresponding to $F^\dt$  and $B(\ZZ,G)$ corresponding to  
$\ZZ^\dt${, the cosimplicial group that sends $[n]$ to the abelianization $\ZZ^n$ of $F^n$}.
\revfive{Note that this means that the topological space $B(\tau, G)$ depends on the map $\eta^\bullet$. However, this is not signaled in notation in prior work \cite{ACT12}; we shall follow this convention.}

We carry over this construction to the simplicial category. \revfour{Let $K$ be a simplicial group.}
The simplicial set $\Wbar K$ is isomorphic to the total simplicial set $TNK$ of the nerve of $K$
{\cite{Ste12}*{Lemma 15}}. 
The nerve $NK$ is the \revfour{horizontally reduced} bisimplicial set 
{$(NK)_{p,q}=N(K_q)_p\cong\Hom(F^p,K_q)$.}
{Our $\tau$ version can be obtained by modifying} 
this bisimplicial set {by setting} 
{$N(\tau,K)_{p,q}=\Hom(\tau^p,K_q)$.}
Then  
{the $\tau$ version of the bar construction is defined by}
$\Wbar(\tau,K)=TN(\tau,K)$.

\revfour{We also generalize the important property that t}he original bar construction \revfour{$\Wbar\colon\catsGrp\to\catsSet$} has a left adjoint \revfour{$G\colon\catsSet\to\catsGrp$}, Kan's loop group functor, see e.g.~\cite{GJ99}*{Lemma 5.3}.
It is shown in \cite{Ste12}*{Proposition 16} that the \revfour{adjunction $\Adj{G}{\catsSet}{\catsGrp}{\Wbar}$ factors as the composite of adjunctions}
$$
\adjcomp{\catsSet}{\catssSet}{\catsGrp}{\Dec}{T}{\pi_1}{N}
$$
\revfour{where $\Dec$ is the d\'ecalage functor and the}
\revfive{functor $\pi_1:\catssSet\to\catsGrp$ is induced naturally by the left Kan extension $\pi_1:\catsSet\to\catGrp$ of the free cosimplicial group $F^\bullet:\Delta\to\catGrp$ along the the natural inclusion 
$\Delta^\bullet \colon \Delta \longrightarrow \catsSet$ that sends $[n]$ to $\Delta[n]$.

For a reduced simplicial set $X$, the group $\pi_1X$ is the usual simplicial fundamental group on the unique vertex of $X$; on a general simplicial set $X$, it is the fundamental group of the quotient $X/\sk_0X$, where the skeleton $\sk_0X$ is the smallest simplicial subset of $X$ that contains the vertices $X_0$.
}

\revfour{
{We let $\pi_1(\tau, -)$ denote the left Kan extension of the cosimplicial group $\tau^\bullet$ along $\Delta^\bullet$. It is by construction a left adjoint of the modified nerve functor $N(\tau,-)$.
Finally, this gives us the modified loop group functor $G(\tau,-)=\pi_1(\tau,-)\circ\Dec$, together with an adjunction $G(\tau,-)\adjoint\Wbar(\tau,-)$.}
{We give a self-contained introduction to d\'ecalage and the loop group functor in Section \ref{sec:loop-decalage}, and develop our modifications in Section \ref{sec:varWbar}.}
}

We compare our construction to the topological version.
 
\begin{customthm}{\ref{thm:B-Wbar}}
There is a natural homotopy equivalence
$$
B(\tau,|K|) \rightarrow |\Wbar(\tau,K)|.
$$ 
\end{customthm}

Thus our construction has the correct homotopy type. We also show that

\begin{customcor}{\ref{cor:simp-Lie}}
 For a compact Lie group $G$, there is a natural homotopy equivalence
$$
B(\tau,G) \to |\Wbar(\tau,SG)| 
$$
where $S$ denotes the singular functor, the right adjoint of the geometric realization.
\end{customcor}

Consequently, we can import the results regarding all the interesting examples studied previously, such as \cites{AG15,AGV17,AGLT17,O16,GH19,OW19}, to the simplicial category.

Our simplicial construction has the advantage of making available methods from simplicial homotopy theory that are not available in the original topological construction. 
{They are motivated by the flexibility of simplicial theory, which naturally accommodates various constructions, such as those in \cite{Vil17} that arise from different cosimplicial structures.}
{On the other hand,  
{using the fact that}
geometric realization is part of a Quillen equivalence and {applying}  Theorem \ref{thm:B-Wbar}, we 
{obtain the following}
bijections:
$$
[B,\Wbar(\tau,K)]\cong[|B|,|\Wbar(\tau,K)|]\cong[|B|,B(\tau,|K|)].
$$
Therefore principal $|K|$-bundles over $|B|$ with commutativity structure are classified by
\revfive{the homotopy classes $[B,\Wbar(\tau,K)]$.}
\revfive{Since $\Wbar(\tau,K)$ is not a Kan complex in general, these homotopy
classes are computed by
replacing $\Wbar(\tau,K)$ with a fibrant (Kan) replacement. It is not straighforward to describe representatives of these homotopy classes of maps or the kind of principal $|K|$-bundles they classify.}
The authors plan to pursue a simplicial study of principal bundles with commutativity structure in future work.

{Recently, simplicial versions of commutative classifying spaces have been useful in the theory of simplicial distributions \cites{okay2022simplicial,okay2024twisted}, a framework for constructing probability distributions on principal bundles to analyze distributions that arise in quantum theory. 
We expect that our generalizations will yield new and interesting examples of simplicial distributions.}
}

{The rest of the paper is organized as follows.
In Section \ref{sec:loop-decalage}, we study the total d\'ecalage  of the standard $n$-simplex. Results in this section are used to define the modified $\Wbar$-construction.
In Section \ref{sec:varWbar} we introduce $\Wbar(\tau,K)$ and show that after looping once the canonical map to $\Wbar(K)$ splits up to homotopy (Proposition \ref{pro:splitting}). We compare our simplicial construction to the original topological construction in Section \ref{sec:comptopcat}. Our main result Theorem \ref{thm:B-Wbar} is proved in this section. 
}

{\bf Acknowledgements:}
We would like to thank the editor, Haynes Miller, for very helpful questions and comments, the anonymous referee for valuable comments, Danny Stevenson for sharing a proof of Corollary \ref{cor:G=pi1Dec}, and Ben Williams for feedback on an earlier version of this paper.
The first author is  supported by the Air Force Office of Scientific
Research under award number FA9550-21-1-0002; and would like to thank Alfr\'ed R\'enyi Institute of Mathematics for their hospitality during a visit in the summer of 2019. The second author is partially supported by the project NKFIH K 138828.

\section{ Kan's loop group and d\'{e}calage} \label{sec:loop-decalage}

{
In this section our goal is to give an alternative description of Kan's loop group functor.}
The standard $n$-simplices $\Delta[n]$ as $n$ varies can be assembled into a cosimplicial simplicial set $\Delta^\dt$. Applying Kan's loop group functor $G$ level-wise gives a cosimplicial simplicial group. In this section we describe this object using the factorization   $G=\pi_1 \Dec$.
Proposition \ref{pro:Loop-pi1Dec-iso} gives an explicit morphism $G\Delta^\dt \rightarrow \pi_1\Dec\, \Delta^\dt$ of cosimplicial simplicial groups.
This description will be essential later on when we introduce variations of the $\Wbar$-construction.  For the properties of the loop group functor and the d\'{e}calage functor we refer to \cite{Ste12}.

For a category $\catC$ we write $\mathbf{s}\catC$ for the category of simplicial objects in that category.

\subsection{D\'{e}calage }

Let $\catDelta$ denote the simplex category.
Let $\catDelta_+$ denote the category obtained from $\catDelta$ by adjoining the empty ordinal denoted by $[-1]$.
The category $\catDelta_+$ is a monoidal category with unit $[-1]$. The monoidal structure is given by the bifunctor
$
+:\catDelta_+\times \catDelta_+\rightarrow \catDelta_+
$
which sends a pair $([m],[n])$ to the ordinal $[m+n+1]$, and  a pair of morphisms $\varphi:[m]\rightarrow [m']$ and $\theta:[n]\rightarrow [n']$ to the morphism
$$
(\varphi+\theta)(i)=\left\lbrace
\begin{array}{ll}
\varphi(i) & 0\leq i\leq m \\
\theta(i-m-1)+m'+1 & m+1\leq i \leq m+n+1
\end{array}
\right.
$$

Let $\catsSet$ ($\catsSet_+$) denote the category of (augmented) simplicial sets. Given a simplicial set $X$  the augmented simplicial set $\Dec_0 X$ is defined by pre-composing $X$ with the functor $\catDelta_+ \rightarrow \catDelta$ defined  by $[n]\mapsto [n]+[0]$.
\revfour{Here $[n]+[0]$ refers to the monoidal product in $\catDelta_+$,}
\revfour{which we defined above explicitly,}
\revfour{viewed as an object of $\catDelta$,}
\revfour{by restriction to the poset of nonempty ordinals.}
 Note that $(\Dec_0 X)_n=X_{n+1}$ and the simplical structure maps $d_i:(\Dec_0X)_n\to (\Dec_0X)_{n-1}$ and $s_j:(\Dec_0X)_n \to (\Dec_0X)_{n+1}$ are given by
$$
d_i:X_{n+1}\to X_n,\;\;\;\;  s_j:X_{n+1}\to X_{n+2}
$$
where $0\leq i,j\leq n$.
$\Dec_0 X$ is an augmented simplicial set where the augmentation map $d_0:\Dec_0X\rightarrow X_0$  is induced in degree $n$ by  the $(n+1)$-fold composition \revC{$(d^0)^{n+1}=d^0\circ\cdots\circ d^0:[0]\rightarrow [n]+[0]$}.

Let us describe $\Dec_0(\Delta[k])$ in more detail.

\begin{lem}\label{lem:Dec0Delta}
There is an  isomorphism of augmented simplicial sets
\begin{equation}\label{Dec0}
\begin{tikzcd}
\coprod_{0\leq l \leq k} \Delta[l] \arrow{r}{\cong} \arrow{d} &  \Dec_0(\Delta[k]) \arrow{d}{d_0}\\
\set{0,1,\cdots,k} \arrow[r,equal] & \Delta[k]_0
\end{tikzcd}
\end{equation}
which is natural with respect to morphisms in $\catDelta$ and maps $\Delta[l]$ isomorphically onto the pre-image $(d_0)^{-1}(l)$.
\end{lem}
\Proof{
$\Dec_0(\Delta[k])$ is the coproduct  of $(d_0)^{-1}(l)$ over $l$ since the augmentation map $d_0$ is a deformation retraction \cite{Ste12}.
An $m$-simplex of $(d_0)^{-1}(l)$ is a functor $\varphi:[m+1]\rightarrow [k]$ such that $\varphi d_0^{m+1}(0)=l$, and such functors are in one-to-one correspondence with functors $[m]\rightarrow [l]$.
}

\subsection{Total d\'{e}calage}
The d\'{e}calage functor  $\Dec_0$ is a comonad whose structure maps are described as follows: $\Dec_0\rightarrow(\Dec_0)^2  $
is induced by $[n]+[0]+[0]\rightarrow [n]+[0]$ given by the sum of the identity map on $[n]$ and $s^0:[0]+[0]\rightarrow [0]$, and the other structure map $\Dec_0 \rightarrow \idy$ is induced by  \revC{$(d^0)^n:[n]\rightarrow [n]+[0]$}. The total d\'{e}calage  of the simplicial set $X$ is defined to be the \revC{bisimplicial set $\Dec X$ obtained from the} comonadic resolution \revC{of $X$, which in degree $n$, is given by the simplicial set}
$$
\Dec_n X =(\Dec_0)^{n+1} X.
$$
If we think of  $\Dec X$   as a vertical bisimplicial set with horizontal simplicial sets then the set of $(m,n)$-simplices is $X_{m+n+1}$. The horizontal face and degeneracy maps are given by $d_i^h=d_i:(\Dec_nX)_m\rightarrow (\Dec_nX)_{m-1}$ and $s_i^h=s_i:(\Dec_nX)_m\rightarrow (\Dec_nX)_{m+1}$ where $0\leq i\leq m$, the vertical face and degeneracy maps are given by  $\revC{d_j^v}=d_{m+1+j}:(\Dec_nX)_m\rightarrow (\Dec_{n-1}X)_{m}$ and $\revC{s_j^v}=s_{m+1+j}:(\Dec_nX)_m\rightarrow (\Dec_{n+1}X)_{m}$ where $0\leq j\leq n$.

We will give a description of $\Dec\,\Delta[k]$. For this we introduce a bisimplicial set $D[k]$ defined by
$$
D_n[k] =  \coprod_{\varphi:[n]\to [k]} \Delta[\varphi(0)]
$$
together with the simplicial structure
$$
d_i(\varphi,\theta) = \left\lbrace \begin{array}{cc}
(\varphi d^0, \iota \theta) & i=0  \\
(\varphi d^i , \theta) & 0<i \leq n

\end{array} \right.
$$
where $\iota:\Delta[\varphi(0)]\to \Delta[\varphi(1)]$ is induced by the inclusion $[\varphi(0)]\subset [\varphi(1)]$ and
$$
s_j (\varphi,\theta)  =(\varphi s^j,\theta)
$$
for all $0\leq j\leq n$.

\Pro{\label{pro:DecDelta}
There is an isomorphism of bisimplicial sets
$$
g: D[k]\to \Dec\,\Delta[k].
$$
}
\begin{proof}
Applying   Lemma \ref{lem:Dec0Delta}   gives an isomorphism
\begin{equation}\label{dec_Del}
g_n:\coprod_{\varphi:[n]\rightarrow [k]} \Delta[\varphi(0)] \rightarrow \Dec_n\Delta[k]
\end{equation}
which  in degree $m$ sends $(\varphi,\theta)$ to the functor $[m+n+1]\rightarrow [k]$ defined using $\theta$ on the subset $[m]\subset [m+n+1]$ and on the rest  using $\varphi$. More precisely it is the unique functor which fills the diagram
$$
\begin{tikzcd}
{[m]}\arrow[d,hook] \arrow{rd}{\theta} & \\
{[m+n+1]} \arrow[r,dashed] & {[k]} \\
{[n]} \arrow{u}{(d^0)^{m+1}} \arrow[ru,"\varphi"']
\end{tikzcd}
$$
On the other hand the inverse to this map is given by
\begin{equation}\label{dec_Del_inv}
 g_n^{-1}: \Dec_n\Delta[k] \rightarrow \coprod_{\varphi:[n]\rightarrow [k]} \Delta[\varphi(0)]
\end{equation}
which in degree $m$ sends a functor $\alpha: [m+1+n]\rightarrow [k]$ to the pair of the functors given by $\alpha (d^0)^{m+1}:[n]\rightarrow [k]$ and $\alpha':[m]\rightarrow [\alpha(m+1)]$ that fits into the diagram
$$
\begin{tikzcd}
{[m+1+n]} \arrow{r}{\alpha} & {[k]}\\
{[m]} \arrow[hook,u] \arrow{r}{\alpha'} & {[\alpha(m+1)]} \arrow[hook,u]
\end{tikzcd}
$$
It remains to check that $\set{g_n}_{n\geq 0}$ gives a morphism of simplicial sets. It is straightforward to check this for $s_j$ and $d_i$ with $i>0$. When $i=0$ the face map $d_0$ changes $\varphi(0)$ to $\varphi(1)$, and this is accounted for by adding the inclusion $\iota$.

\end{proof}

\subsection{Nerve functor and its adjoint}
We can take the nerve of a group to obtain a simplicial set. Let $\catGrp$ denote the category of groups. This construction gives a functor $N:\catGrp \rightarrow \catsSet$. In fact, the image of the functor lives in the category $\catsSet_0$ of reduced simplicial sets {(i.e., those simplicial sets whose set of  $0$-simplices is a singleton)}. Let us define  a functor  $\pi_1:\catsSet \rightarrow \catGrp$ by assigning \revC{to} a simplicial set $X$ the quotient of the free group $F(X_1)$  on the set of $1$-simplices by the relations $[s_0x]=1$ for all $x\in X_0$, and $ [d_2\sigma][d_0\sigma]=[d_1\sigma]$  for all $\sigma\in X_2$. The functor $\pi_1$ is the left adjoint of $N$. \revfour{Note that for reduced simplicial sets our functor $\pi_1$ agrees with the usual fundamental group. In general,}
\revfive{we first take the quotient by the skeleton $\sk_0$}
\revfour{and then take the usual fundamental group.}

Let us compute $\pi_1$ of a $k$-simplex. Note that the set of $1$-simplices is given by functors $[1]\rightarrow [k]$. The relation coming from $s_0$ will kill those which are not injective. Each $2$-simplex will introduce a  relation among the edges in its boundary. As a result there is an isomorphism of groups
$$
F^k \rightarrow \pi_1 \Delta[k]
$$
which sends a generator $e_i$ of the free group $F^k$ to the $1$-simplex $[1]\rightarrow [k]$ specified by the image $\set{i-1,i}$. We will  identify these groups and given an ordinal map $\alpha:[k]\rightarrow [l]$ we will write
$$
\alpha_*: F^{k} \rightarrow F^l
$$
for the induced map $\pi_1\Delta[k]\rightarrow \pi_1\Delta[l]$.
Considering all the simplices at once defines a cosimplicial group $F^\dt:\catDelta \rightarrow \catGrp$ where $F^\dt[k]$ is the free group $F^k=\pi_1\Delta[k]$. At this point we also observe that the nerve functor is represented by the cosimplicial group $F^\dt$ in the sense that
$$(NG)_\dt=\catGrp(F^\dt,G).$$
 This approach allows us to describe the adjunction $\pi_1\adjoint N$ from the general theory of Kan extensions.

\subsection{Loop group}
 Given a (reduced) simplicial set $X$ the loop group $GX$ is a simplicial group which homotopically behaves as the loop space. The group of $n$-simplices is the free group {$F(X_{n+1})/F(s_0X_n)$}.
The face maps   are defined by
$$
d_i[x] =\left\lbrace \begin{array}{ll}
{[d_1x][d_0x]^{-1}} & i=0\\
{[d_{i+1}x]} & 0< i\leq n
\end{array}\right.
$$
and the degeneracy maps are defined by $s_j[x]=[s_{j+1}x]$ for all $0\leq j\leq n$. In fact this construction is related to the d\'{e}calage construction.

Recall that the bisimplicial set $\Dec\, \Delta[k]$ has a nice description whose  simplicial structure is given in Proposition \ref{pro:DecDelta}. Let us consider the simplicial group $ \pi_1\Dec \Delta[k] $ obtained by applying $\pi_1$ to each vertical degree. From the isomorphism \ref{dec_Del} we see that the resulting simplicial group consists of the free groups
\begin{equation}\label{eq:pi1Dec}
(\pi_1\Dec \Delta[k])_n \cong (\pi_1D[k])_n = \coprod_{\varphi:[n]\rightarrow [k]} F^{\varphi(0)}
\end{equation}
whose simplicial structure maps are induced by the ones of $D[k]$. Let us describe this structure explicitly. We write $[\varphi,e_j]$ to denote the generator of $F^{\varphi(0)}$ corresponding to the $\varphi$ term in the coproduct. The simplicial structure maps are given by
\begin{equation}\label{eq:pi1Dec-simplicial}
d_i[\varphi,e_j] =\left\lbrace \begin{array}{ll}
{[\varphi d^0, \iota_*(e_j)]} & i=0 \\
{[\varphi d^i, e_j]} & 0<i\leq n
\end{array} \right.
\end{equation}
and $s_i[\varphi,e_j]=[\varphi s^i,e_j]$ for the degeneracy maps. Recall that $\iota_*$ is the map $\pi_1 \Delta[\varphi(0)]\to \pi_1 \Delta[\varphi(1)]$ induced by the inclusion $[\varphi(0)]\subset [\varphi(1)]$.

\subsection{Relation to d\'{e}calage}
 Let $\Delta^\dt:\catDelta \rightarrow \catsSet$ denote the functor defined by sending $[k]$ to the $k$-simplex $\Delta[k]=\catDelta(-,[k])$. We can think of $\Delta^\dt$ as  a cosimplicial object in the category of simplicial sets.
Applying the functor $\Dec$ to the cosimplicial object $\Delta^\dt$ in each degree and taking $\pi_1$ of the resulting simplicial set gives a cosimplicial object
$$
\pi_1\Dec\,\Delta^\dt:\catDelta \to \catsGrp.
$$
The object $[n]$ is sent to the \revC{simplicial} group $\pi_1\Dec\,\Delta[n]$. Another cosimplicial object can be obtained by composing with Kan's loop group functor
$$
G\Delta^\dt:\catDelta \to \catsGrp.
$$

\Pro{\label{pro:Loop-pi1Dec-iso} There is a natural isomorphism of  functors
$$
\epsilon^\dt:G\Delta^\dt \rightarrow \pi_1\Dec\, \Delta^\dt.
$$ }
\begin{proof}
\revC{Consider the cosimplicial degree $k$ and simplicial degree $n$. We begin by defining a group homomorphism
$
(\epsilon^k)_n: G(\Delta[k])_n \to (\pi_1 \Dec \Delta[k])_n.
$
For this recall that $G(\Delta[k])_n$ is a quotient of the free group on the set $\Delta[k]_{n+1}$. By \ref{eq:pi1Dec} $(\pi_1 \Dec \Delta[k])_n$ is the free product of $ \pi_1(\Delta[\varphi(0)])$ over the ordinal maps $\varphi:[n]\rightarrow [k]$, hence a free group on the disjoint union $\coprod_\varphi \Delta[\varphi(0)]_1$. We define a map between the set of generators of the free groups}
$$
\Delta[k]_{n+1} \rightarrow \coprod_{\varphi:[n]\rightarrow [k]} \Delta[\varphi(0)]_1
$$
by sending $\beta:[n+1]\rightarrow [k]$ to the pair given by $\beta d^0:[n]\rightarrow [k]$ and $\beta|_{[1]}$, the restriction of $\beta$ along the natural inclusion $[1]\subset [n+1]$.
\revC{This map induces a group hommorphism}
$$
(\epsilon^k)_n: G (\Delta[k])_n \longrightarrow \revC{(\pi_1\Dec \Delta[k])_n}
$$
since a generator $[\alpha s^0]$ \revC{for some $\alpha:[n+1]\to [k]$} in the image of the degeneracy map $s_0$ is mapped to the generator given by the pair of the map $\alpha s^0 d^0=\alpha$ and the map obtained by restricting $\alpha s^0$ to $[1]$. The latter generator corresponds to $[(d^0)^{\alpha(0)}s^0]$ in the free group $F^{\alpha(0)}$ and is equivalent to the identity element under the identifications \revC{in $ \pi_1(\Delta[\varphi(0)])$}.
Moreover, $(\epsilon^k)_n$ is an isomorphism since both groups are free groups and the map is a bijection between the generators.
 For simplicity of \revC{notation} we will suppress the cosimplicial degree and write $\epsilon_n=(\epsilon^k)_n$.
Next we check that the map is compatible with the simplicial structure. For $0<i\leq n$ we have $\epsilon_nd_i(\alpha)=\epsilon_n(\alpha d^{i+1})=\alpha d^{i+1}d^0=\alpha d^0d^{i}$, similarly for degeneracy maps we have $\epsilon_ns_j(\alpha)= \alpha d^0s^{j}$ for all $0\leq j\leq n$. Thus for the face map we have
$$
\epsilon_n(d_i[\alpha])= [\alpha d^0 d^i , \alpha d^{i+1} |_{[1]}] = d_i[\alpha d^0 , \alpha|_{[1]}] = d_i\revC{\epsilon_{n+1}}[\alpha]
$$
and similarly $\epsilon_n$ commutes with the degeneracy maps.
Finally $d_0$ on the loop group is defined by $d_0[\alpha]=[\alpha d^1][\alpha d^0]^{-1}$.
Therefore we have
$$
\epsilon_n(d_0[\alpha])=\epsilon_n([\alpha d^1])\epsilon_n([\alpha d^0])^{-1}=[\alpha (d^0)^2,\alpha d^1 |_{[1]}]\;\revC{[\alpha (d^0)^2,\alpha d^0 |_{[1]}]^{-1}}
$$
where we used the simplicial identity  $d^1d^0=(d^0)^2$.
\revC{U}nder $\epsilon_n$ both $[\alpha d^1]$ and $[\alpha d^0]$ map to the same free group $F(\alpha(2))$ indexed by $\alpha (d^0)^2$.
Let $\sigma$ denote the restriction of $\alpha$ to $[2]$ (note that $\alpha$ lives in degree $\geq 2$). Then using the relation in the fundamental group  $\pi_1 \Delta[\alpha(2)]\cong F(\alpha(2)) $ we can write
$$
[\alpha d^1 |_{[1]}][\alpha d^0 |_{[1]}]^{-1} = [d_1\sigma][d_0\sigma]^{-1} = [d_2\sigma] = [\alpha|_{[1]}]
$$
which is the same, after adding the coproduct index, as $d_0$ of the generator \revC{$\epsilon_{n+1}[\alpha]$}.

 Thus   we showed that we have an isomorphism of simplicial groups
$$
\epsilon^k : G \Delta[k] \rightarrow \pi_1 \Dec \Delta[k]
$$
where we remembered the cosimplicial index, and suppressed the simplicial index. The resulting map is compatible with the cosimplicial structure since all the constructions involved are functorial in $\Delta[k]$.
\end{proof}

\Rem{{\rm \label{rem:inverse-iso}
One can check that the inverse of $\epsilon^k$ is induced by the map
\begin{equation}\label{map:epsilon-inv}
 \coprod_{\varphi:[n]\rightarrow [k]} \Delta[\varphi(0)]_1 \rightarrow \Delta[k]_{n+1}
\end{equation}
which sends $(\varphi,\gamma)$ to the functor $\alpha:[n+1]\rightarrow [k]$ defined by $\alpha(0)=\gamma(0)$ and $\alpha(i)=\varphi(i-1)$ for $0<i\leq n+1$. An argument similar to the proof of Proposition \ref{pro:Loop-pi1Dec-iso} can be used to see that after taking the appropriate quotients of the free groups the resulting map is an isomorphism.
We will denote this map simply as the inverse
 $$
(\epsilon^\dt)^{-1}:\pi_1\Dec \Delta^\dt \rightarrow G\Delta^\dt.
$$
}}

Almost immediately we obtain the following result which is also proved in \cite[Proposition 16]{Ste12}, however, without an explicit isomorphism.

\Cor{\label{cor:G=pi1Dec}For any simplicial set $X$ there is a natural isomorphism of simplicial groups
$$
G(X) \to \pi_1\Dec\,(X).
$$
}
\Proof{
An arbitrary simplicial set can be written as a colimit of $\Delta[n]$ over the simplex category. All the functors in sight $G$, $\pi_1$, \revC{and $\Dec$ (see \cite[\S 3]{Ste12})} are left adjoints, thus they preserve colimits. Then the result follows from Proposition \ref{pro:Loop-pi1Dec-iso}.
}

\Rem{{\rm In Proposition 16 of \cite{Ste12} Kan's loop group $G$ is compared to $\pi_1 R\, \Dec$ where $R$ is the left adjoint of the inclusion $\catssSet_0 \to \catssSet$. We avoided this functor by defining $\pi_1$ for an arbitrary simplicial set not necessarily a reduced one.
}}

\subsection{Simplicial group homomorphisms}
Let $K$ be a simplicial group. The set of simplicial group homomorphisms $G\Delta[k]\to K$ is well-known \cite{GJ99}.
This set is precisely the set   of $k$-simplices of the $\Wbar$-construction of $K$. For variations of this construction we need an explicit description of such simplicial group homomorphisms.

By Proposition \ref{pro:Loop-pi1Dec-iso} we can instead consider a morphism
 $f:\pi_1\Dec \Delta[k]\rightarrow K$   of simplicial groups. The degree-$n$ part $f_n$  belongs to the set
$$
\catGrp(\pi_1\Dec\Delta[k]_n,K_n) = \prod_{\varphi:[n]\to[k]} \catGrp(F^{\varphi(0)},K_n) = \prod_{\varphi:[n]\to[k]} K_n^{\times \varphi(0)}
$$
where $K_n^{\times \varphi(0)}$ is understood to be the trivial group if $\varphi(0)=0$. Therefore $f_n$ is determined by the tuple of elements $f_{n}[\varphi,e_j]$ in $K_n$ where  $e_j\in F^{\varphi(0)}$ and $1\leq j\leq \varphi(0)$ . The map $\varphi$ factors as follows
$$
\begin{tikzcd}
{[n]} \arrow{r}{\varphi} \arrow[rd,"\phi"'] & {[k]} \\
& {[k-\varphi(0)]} \arrow[hook,u,"(d^0)^{\varphi(0)}"']
\end{tikzcd}
$$
where $\phi(0)=0$ i.e. the canonical decomposition of $\phi$ does not involve $d^0$. The simplicial structure in \ref{eq:pi1Dec-simplicial}  implies that
$$
\phi^*[(d^0)^{\varphi(0)},e_j] = [(d^0)^{\varphi(0)}\phi,e_j] = [\varphi,e_j].
$$
and thus $f_n[\varphi,e_j]=\phi^*f_{k-\varphi(0)}[(d^0)^{\varphi(0)},e_j]$. As a result   the elements $f_{k-l}[(d^0)^l,e_j]$ where $e_j\in F^l$ for $l=1,\cdots,k$ completely determine $f_n$. It remains to consider the effect of $d^0$. As a consequence of \ref{eq:pi1Dec-simplicial} we have
$$
d^0[(d^0)^l,e_j] = [(d^0)^{l+1},d^{l+1}_*e_j].
$$
Note that in this case $\iota$ is induced by the inclusion $d^{l+1}:[l]\to [l+1]$. By abuse of notation we will identify $d^{l+1}_*e_j$ with the generator $e_j$ in $F^{l+1}$. After this identification we see that all the other generators are determined once we fix $f_{k-l}[(d^0)^l,e_l]$ for $1\leq l\leq k$, where $e_l$ is the top generator of $F^l$. We package this information as a diagram
\begin{equation}\label{diag:SGHom}
\begin{tikzcd}
F^1  \arrow{r}{d_*^2} \arrow{d}{f_{k-1}} & F^2  \arrow{r}{d_*^3} \arrow{d}{f_{k-2}} & \cdots  \arrow{r}{d^k_*} & F^k  \arrow{d}{f_{0}} \\
K_{k-1}  \arrow{r}{d_0}  & K_{k-2} \arrow{r}{d_0}   & \cdots \arrow{r}{d_0}  & K_0
\end{tikzcd}
\end{equation}
where $d^l_*$ is the map induced by $d^l:[l-1]\to [l]$.

Evaluating the homomorphisms $f_{k-l}$ at the generator $e_l$ of each free group gives a $k$-tuples
$
(x_{k-1},x_{k-2},\cdots, x_0)
$
where $x_{k-l}$ belongs to $K_{k-l}$. Then   the images of the generators of $F^l$ are given by
\begin{equation}\label{eq:imagesofgen}
f_{k-l}(e_j)=(d_0)^{l-j}x_{k-j}
\end{equation}
where $1\leq j\leq l$.

\Pro{\label{pro:SGHomSet}
There is a bijection of sets
$$
\catsGrp(\pi_1\Dec\Delta[k],K) \longrightarrow K_{k-1}\times K_{k-2}\times \cdots \times K_0.
$$
defined by sending
$$
f\mapsto (x_{k-1},x_{k-2},\cdots, x_0)
$$
where $x_{k-l}=f_{k-l}[(d^0)^l,e_l]$.
}

\Rem{{\rm
In particular, for $K=G\Delta[k]$ we can find the $k$-tuple corresponding to  the map $(\epsilon^\dt)^{-1}$ defined in Remark \ref{rem:inverse-iso}. Consider the restriction $\Delta[l]_1 \to \Delta[k]_{n+1}$ of the map in \ref{map:epsilon-inv} to the factor $\varphi=(d^0)^l:[k-l]\to [k]$ in the coproduct.
The generator $e_l$ in $F^l$ is represented by the map $(d^0)^{l-1}:[1]\to [l]$. Under \ref{map:epsilon-inv} the pair $((d^0)^l,(d^0)^{l-1})$ is mapped to $(d^0)^{l-1}:[k-l+1]\to [k]$. As a result we obtain
$$
f_{k-l}(e_l) = [(d^0)^{l-1}] \in G\Delta[k]_{k-l}
$$
Thus,  under the bijection in Proposition \ref{pro:SGHomSet} we have
$$
(\epsilon^\dt)^{-1} \mapsto   ([\idy],[d^0],\cdots,[(d^0)^{k-1}]).
$$
}}

\section{Variants of the $\Wbar$-construction}\label{sec:varWbar}

{In this section w}e introduce the simplicial set $\Wbar(\tau,K)$ that depends on an endofunctor $\tau$ on the category of groups.  Under some niceness conditions on $\tau$ we describe the set of $n$-simplices. This object comes with a map $\Wbar(\tau,K)\to \Wbar K$, which we show splits up to homotopy after looping once.

When $\tau$ is the identity functor we recover $\Wbar K$. Other examples come mainly from the descending central series, and in particular, the abelianization functor, which gives a simplicial version of $B(\ZZ,G)$.

\subsection{$\Wbar$-construction} The loop functor $G$ has a well-known right adjoint $\Wbar:\catsGrp \rightarrow \catsSet$. For a simplicial group $K$ the set of $k$-simplices of $\Wbar K$ consists of simplicial group homomorphisms $G\Delta[k]\rightarrow K$, and the simplicial structure is determined by the cosimplicial structure of $G\Delta^\dt$. More explicitly,     $\Wbar(K)_k$  can be identified  with the product $K_{k-1}\times K_{k-2}\times \cdots\times K_0$. Under this identification the simplicial structure is described as follows:   the face maps are given by
{\small
$$
d_i(x_{k-1},\cdots,x_0) = \left\lbrace
\begin{array}{ll}
(x_{k-2},x_{k-3},\cdots,x_0) & \revC{i=0}\\
(d_{i-1}x_{k-1},d_{i-2}x_{k-2},\cdots,d_0x_{k-i}x_{k-i-1},x_{k-i-2},\cdots,x_0 ) & 0<i<k \\
(d_{k-1}x_{k-1},d_{k-2}x_{k-2},\cdots , d_1x_1) & i=k
\end{array}\right.
$$
}
and the degeneracy maps are given by
{\small
$$
s_i(x_{k-1},\cdots,x_0) = \left\lbrace
\begin{array}{ll}
 (1,x_{k-1},\cdots,x_0) & i=0 \\
(s_{i-1}x_{k-1},s_{i-2}x_{k-2},\cdots,s_0x_{k-i},1,x_{k-i-1},x_{k-i-2},\cdots,x_0 ) & 0<i\leq k.
\end{array}\right.
$$
}

As a consequence of the natural isomorphism $G\Delta^\dt \rightarrow \pi_1\Dec\Delta^\dt$    proved in Proposition \ref{pro:Loop-pi1Dec-iso} the functor $\pi_1\Dec$ has also a right adjoint   isomorphic to the functor $\Wbar$.
The right adjoint of $\pi_1\Dec$ is defined by
$$
K \mapsto \catsGrp( \pi_1\Dec\Delta^\dt, K)
$$
where the simplicial group homomorphisms are taken level-wise. Let us describe the cosimplicial structure of $\pi_1\Dec\Delta^\dt$. Given an ordinal map $\alpha:[l]\rightarrow [k]$ the induced map on d\'{e}calage $\Dec \Delta[l] \rightarrow \Dec\Delta[k]$ sends a pair $(\varphi,\theta)$ to composition by $\alpha$, namely to the pair $(\alpha\varphi,\alpha\theta)$. The map between the fundamental groups is determined when $\theta$ is a $1$-simplex of $\Delta[\varphi(0)]$, that is, we have $\alpha^*[\varphi,e_j]=[\alpha\varphi,\alpha_*(e_j)]$ where $e_j$ belongs to the free group $F^{\varphi(0)}$. Given this general description for arbitrary ordinal maps  let us figure out the effect of the coface maps $d^i:[k-1]\rightarrow [k]$ first. It suffices to consider the generators $[(d^0)^m,e_{m}]$, where $1\leq m\leq k-1$, since the rest is determined by the simplicial structure. Using the cosimplicial identity $d^id^0=d^0d^{i-1}$ for $i>0$ we obtain
\begin{equation}\label{eq:DecDelta-face}
d^i[(d^0)^m,e_{m}] =\left\lbrace
\begin{array}{ll}
{ [(d^0)^m d^{i-m},e_{m}]} & 1\leq m < i \leq k \\
{ [(d^0)^{m+1},e_m]}\,{ [(d^0)^{m+1},e_{m+1}]} & m=i \\
{[(d^0)^{m+1},e_{m+1}]} & i< m\leq k-1
\end{array}
\right.
\end{equation}
and for the codegeneracy maps $s^i:[k+1]\rightarrow [k]$, using $s^id^0=d^0s^{i-1}$ for $i>0$ and $s^0d^0=1$, we have
\begin{equation}\label{eq:DecDelta-degeneracy}
s^i[(d^0)^m,e_{m}] =\left\lbrace
\begin{array}{ll}
{ [(d^0)^{m}s^{i-m},e_{m}]} & 1\leq m \leq i \leq k  \\
{ [(d^0)^{m-1},1]} & m=i+1\\
{ [(d^0)^{m-1},e_{m-1}]} & i+1 < m\leq k+1.
\end{array}
\right.
\end{equation}

\Pro{\label{pro:rightadjoint-Wbar} There is a natural identification of simplicial sets
$$
\Wbar(K) = \catsGrp( \pi_1\Dec\,\Delta^\dt, K) .
$$
}
\Proof{
Using the cosimplicial structure of $\pi_1\Dec\Delta^\dt$  described in \ref{eq:DecDelta-face} and \revC{\ref{eq:DecDelta-degeneracy}} we check that the right adjoint is  equal to the $\Wbar$-construction.
In Proposition \ref{pro:SGHomSet} we have seen that simplicial group homomorphisms $f:\pi_1\Dec\Delta[k]\rightarrow K$ are in one-to-one correspondence with $k$-tuples $
(x_{k-1},x_{k-2},\cdots, x_0)
$
where $x_i\in K_{i}$. The correspondence is obtained by letting $x_{k-m}$ denote the image of $[(d^0)^m,e_m]$ under $f$. Using the cosimplicial structure of $\pi_1\Dec\Delta^\dt$ we first compute the coface maps $d^i:[k-1]\rightarrow [k]$:
$$
f(d^i{[(d^0)^m,e_m]}) = \left\lbrace \begin{array}{ll}
d_{i-m} x_{k-m} & 1\leq m < i \leq k \\
d_0 (x_{k-i}) x_{k-i-1} & m=i \\
x_{k-m-1} & i<m \leq k-1
\end{array}\right.
$$
whereas the codegeneracy maps give
$$
f(\revC{s^i}{[(d^0)^m,e_m]}) = \left\lbrace \begin{array}{ll}
s_{i-m} x_{k-m} & 1\leq m\leq i \leq k\\
1 & m=i+1 \\
x_{k-m+1} &   i+1< m \leq k-1.
\end{array}\right.
$$
This shows that the simplicial structure on the $k$-tuples $
(x_{k-1},x_{k-2},\cdots, x_0)
$
is exactly the one of the $\Wbar$-construction.
}

\subsection{Endofunctors}
Recall that we used the cosimplicial group $F^\dt$, where $F^k = \pi_1\Delta[k]$, in the definition of the nerve functor, namely $N=\catGrp(F^\dt,-)$. We will introduce a variant of this construction with respect to an endofunctor $\catGrp \to \catGrp$.

Left Kan extension \cite[Chapter 1]{riehl2014categorical} of \revC{a cosimplicial group} $\tau^\dt:\catDelta\rightarrow \catGrp$ along the natural inclusion   $\Delta^\dt:\catDelta \rightarrow \catsSet$  gives a functor
$$
\pi_1(\tau,-):\catsSet \rightarrow \catGrp.
$$
By the general theory of left Kan extensions there is  a corresponding right adjoint
$$
N(\tau,-):\catGrp \rightarrow \catsSet
$$
defined by $N(\tau,G)_n = \catGrp(\tau^n,G)$.
Given an endofunctor $\tau$ we consider the cosimplicial group $\tau^\dt$ defined by
$$
\tau^k =  \tau F^k.
$$
The group $\pi_1(\tau,\Delta[k])$ is naturally isomorphic to $\tau F^k=\tau\pi_1\Delta[k]$.
Note that we recover the adjunction $\pi_1\adjoint N$ when $\tau$ is the identity functor.

\Def{\rm{\label{def:quotient-type}
We say that \emph{$\tau$ is of quotient type}, if there exists a natural transformation $\eta:\idy\to\tau$ such that the map of groups $\eta_{F^n}:F^n\to\tau^n$ is surjective for all $n\ge0$. In this case we also say that \emph{$\eta$ is surjective on finitely generated free groups}.
}}

Let us give a list of endofunctors that are of interest to us, see for example \cite{O14}.
\begin{itemize}
\item Descending central series endofunctor $\Gamma_q$:   The descending central series of a group $H$ is defined by
$$\Gamma^1(H)=H,\;\;\;\;\Gamma^q(H)=[\Gamma^{q-1}(H),H].$$
We will denote the $q$-th stage  $H/\Gamma^qH$ of the descending central series by $\Gamma_qH$.

\item $\Gamma_2$ will have a special importance. We introduce the notation
$$
\ZZ^\dt = \Gamma_2 F^\dt.
$$
This is a cosimplicial group sending $[n]$ to the free abelian group $\ZZ^n$ of rank $n$. Alternatively we can think of this as the functor $[n]\mapsto H_1(\Delta[n]/\Delta[n]_0,\ZZ)$, where $\Delta[n]/\Delta[n]_0$ is the simplicial set obtained by identifying all the vertices of $\Delta[n]$ (see \cite[\S 2]{OW19}).

\item  Mod-$p$ version $\Gamma_{p,q}$: For a group $H$ mod-$p$ descending central series  is defined   by
$$\Gamma^1_p(H)=H,\;\;\;\;\Gamma^q_p(H)=[\Gamma^{q-1}_p(H),H](\Gamma^{q-1}_p(H))^p.$$
The $q$-th stage  $H/\Gamma^q_pH$ is denoted by $\Gamma_{p,q}H$.

\item $\Gamma_{p,2}$ is used to define
$$
(\ZZ/p)^\dt = \Gamma_{p,2} F^\dt.
$$

\item Let $\Gamma_{p^k,2}$ denote the composition of $\Gamma_{2}$ with the mod-$p^k$ reduction functor that sends an abelian group to the largest $p^k$-torsion quotient. We write
$$
(\ZZ/p^k)^\dt = \Gamma_{p^k,2} F^\dt.
$$

\item Let $H\mapsto H^\wedge_p$ denote the $p$-adic completion functor. Let $\hat\Gamma_{p,q}$ denote the functor $H\mapsto \Gamma_q( H^\wedge_p)$.

\item The $p$-adic cosimplicial group is defined by
$$
(\ZZ_p)^\dt = \revC{\hat \Gamma_{p,2}} F^\dt.
$$

\end{itemize}

\Rem{{\rm
The endofunctors $\Gamma_q$ and $\Gamma_{p^k,q}$ are of quotient type. The completed version $\hat\Gamma_{p,q}$ fails to satisfy the surjectivity assumption.
}}

\Ex{{\rm
Let $X(n)$ denote the quotient of $\Delta[n]$ by the set of vertices $\Delta[n]_0$. Then one can check that $\pi_1(\ZZ,-)$ of the inclusion
$$
\vee^n X(1) \rightarrow X(n)
$$
is the abelianization map $F^n\rightarrow \ZZ^n$.  Higher simplices \revC{have} the effect of abelianizing the ordinary fundamental group.
}}

\subsection{Variants of \revtwo{the} $\Wbar$-construction}
We can generalize the adjunction between the $\Wbar$-construction and Kan's loop group functor $G$ with respect to a given endofunctor  on the category of groups.

\revC{Let $\tau^\dt:\Delta \to \catGrp$ be a cosimplicial group.} We define the functor
$$
G(\tau,-):\catsSet \rightarrow \catsGrp\;\;\;\; X\mapsto  \pi_1(\tau, \Dec X)
$$
which, up to natural isomorphism, is the left Kan extension of $\pi_1(\tau, \Dec \Delta^\dt)$ along the inclusion   of the simplex category into the category of simplicial sets.
On a $k$-simplex this functor is given by
$$
G(\tau,\Delta[k])_n = \coprod_{\varphi:[n]\rightarrow [k]} \tau F^{\varphi(0)}
$$
and the simplicial structure is induced from the one of $G \Delta[k]$.

There exists a right adjoint of the functor $G(\tau,-)$ which we denote by
$$
\Wbar(\tau,-):\catsGrp \rightarrow \catsSet.
$$
\revC{We restrict our attention to cosimplicial groups $\tau^\dt$ of the form $\tau F^\dt$ for some endofunctor $\tau: \catGrp \to \catGrp$.} Observe that these constructions recover the usual adjunction $G\adjoint\Wbar$ when $\tau$ is the identity functor.

\Pro{\label{pro:simplicesWbar}
Let $\tau:\catGrp\to\catGrp$ be an endofunctor, and $\eta:\idy\to\tau$ a natural transformation that is surjective on finitely generated free groups (Definition \ref{def:quotient-type}). Then the set of $k$-simplices of $\Wbar(\tau,K)$ is given by tuples
$$
(x_{k-1},x_{k-2},\cdots,x_0) \in K_{k-1}\times K_{k-2}\times \cdots\times K_0
$$
that satisfy the following property: For $1\leq l\leq k$ {\color{black} the homomorphism $f_{k-l}:F^l \to K_{k-l}$ defined by $f_{k-l}(e_j)= (d_0)^{l-j}x_{k-j}$ factors through $\eta_{F^l}:F^l\to \tau F^l$.}
}
\begin{proof}
The set $\Wbar(K)_k$, equivalently the set of simplicial group homomorphisms $\pi_1\Dec\Delta[k]\to K$, is described in Proposition \ref{pro:SGHomSet}.
Since $\eta_{F^n}$ is surjective by assumption the induced map $\eta_*:\catGrp(\tau F^n,H)\to \catGrp(F^n,H)$ is injective for any group $H$.
We see that there is an inclusion of simplicial sets $\Wbar(\tau,K)\subset \Wbar K$ and {\color{black} as a consequence of \ref{diag:SGHom} } the elements of $\Wbar(\tau,K)_k$ \revC{correspond} to diagrams
$$
\begin{tikzcd}
\tau F^1  \arrow{r}{d_*^2} \arrow{d}{f_{k-1}} & \tau F^2  \arrow{r}{d_*^3} \arrow{d}{f_{k-2}} & \cdots  \arrow{r}{d^k_*} & \tau F^k  \arrow{d}{f_{0}} \\
K_{k-1}  \arrow{r}{d_0}  & K_{k-2} \arrow{r}{d_0}   & \cdots \arrow{r}{d_0}  & K_0
\end{tikzcd}
$$

\end{proof}

\Ex{\label{ex:simplices}{\rm
It is instructive to look at $\Wbar(\ZZ,K)$. The set of $k$-simplices consists of
$$
(x_{k-1},x_{k-2},\cdots,x_0)
$$
such that the elements $(d_0)^{l-1} x_{k-1},(d_0)^{l-2}x_{k-2},\cdots,\revC{d_0x_{k-l+1}}, x_{k-l}$
pairwise commute for all $1\leq l\leq k$.

This easily generalizes to $\Wbar(\Gamma_q,K)$.  When $K$ is discrete, i.e. $K_n=G$ for some discrete group $G$ and the simplicial maps are all identity,  the geometric realization of $\Wbar(\Gamma_q,G)$ is precisely the space $B(q,G)$ introduced in \cite{ACT12}.
}}

{
\Rem{
{\rm
In general, $\Wbar(\tau,K)$ is not a Kan complex; that is, it is not fibrant in the Quillen model structure on the category of simplicial sets.
This can be seen by considering $\Wbar(\ZZ,G)$, where $G$ is a discrete non-abelian group  viewed as a simplicial group with identity simplicial structure maps. From Example \ref{ex:simplices}, we see that $n$-simplices of the bar construction can be identified with tuples $(g_1,\cdots,g_n)$ of pairwise commuting elements in $G$. Then, the map $\Lambda_1[2] \to \Wbar(\ZZ,G)$ from the horn that sends the $d_0$ and $d_2$ faces to $g_1$ and $g_2$, where $g_1$ and $g_2$ do not commute, does not extend to $\Delta[2]$. Hence, $\Wbar(\ZZ,G)$ fails to be a Kan complex. 
}}
}

\subsection{Descending central series filtration}
We introduce a simplicial version of the filtration introduced in \cite{ACT12} for the classifying space of a topological group.
This filtration is obtained from the sequence of endofunctors
$$
\revC{\idy \eqqcolon{}} \Gamma_\infty \rightarrow \cdots \rightarrow  \Gamma_q \rightarrow \Gamma_{q-1} \rightarrow \cdots \Gamma_2
$$
associated to the stages of the descending central series. For each endofunctor we have a cosimplicial group
$$
\Gamma^\dt_q = \Gamma_q F^\dt.
$$
We write
$$\Wbar(q,-)=\Wbar(\Gamma_q,-),\;\;\;\;G(q,-)=G(\Gamma_q,-).$$
The resulting sequence
\begin{equation}\label{eq:cofilt-G}
G(X) \rightarrow \cdots \rightarrow  G(q,X) \rightarrow G({q-1},X) \rightarrow \cdots\to G(2,X)=G(\ZZ,X)
\end{equation}
consists of surjective  simplicial group homomorphisms since $\Gamma_q H\to \Gamma_{q-1}H$   is surjective for any group $H$.

On the other hand, we have a sequence of inclusions of simplicial sets
\begin{equation}\label{eq:filt-W}
 \Wbar(\ZZ,K)= \Wbar(2,K) \rightarrow \cdots \rightarrow  \Wbar(q-1,K) \rightarrow \Wbar({q},K) \rightarrow \cdots \to\Wbar(K)
\end{equation}
that yields  a filtration of the $\Wbar$-construction.

Applying \revC{the} $\Wbar$ functor to the sequence of simplicial groups in (\ref{eq:cofilt-G}) gives a cofiltration of $X\simeq \Wbar G(X)$. We can use (\ref{eq:filt-W}) to further filter each term by subspaces of the form $\Wbar(q',G(q,X))$. \revC{The construction} of such filtrations is motivated by \cite[Problem 3]{cohen2016survey}.

\subsection{Kan suspension}   Let $X$ be a pointed simplicial set. \revC{The Kan suspension of $X$} is the simplicial set  $\Sigma X$ whose set of $n$-simplices is given by the wedge $X_{n-1}\vee X_{n-2}\vee \cdots \vee X_{0}$ (\cite[page 189]{GJ99}). An ordinal map $\theta: [m]\rightarrow [n]$ induces $\theta^*:(\Sigma X)_n\rightarrow (\Sigma X)_m$ which maps a wedge summand $X_{n-i}$ to the base point if $\theta^{-1}(i)$ is empty, otherwise it is determined by   $\theta_i^*: X_{n-\theta(i)} \rightarrow X_{m-i}$ where $\theta_i$ is defined by the diagram
\begin{equation*}
\begin{tikzcd}
{[m-i]} \arrow{r}{(d^0)^{i}} \arrow{d}{\theta_i} & {[m]} \arrow{d}{\theta} \\
{[n-\theta(i)]} \arrow{r}{(d^0)^{\theta(i)}} & {[n]}
\end{tikzcd}
\end{equation*}
Let $K$ be a simplicial group pointed by the identity element of each $K_n$.
There is a canonical map of simplicial sets
$$
\kappa: \Sigma K \rightarrow \Wbar K
$$
induced by the inclusion $K_{n-1} \vee \cdots \vee K_{0} \rightarrow K_{n-1} \times \cdots \times K_{0} $ at the $n$-th level.

\Pro{\label{pro:splitting} Assume that $\tau$ is of quotient type and  satisfies the property that $\eta_{F_1}:F_1\to \tau F_1$ is an isomorphism. Then the natural map
$$
G\Wbar(\tau,K) \rightarrow G(\Wbar K)
$$
splits (naturally) up to homotopy.
}

This is a simplicial analogue of \cite[Theorem 6.3]{ACT12} that applies to the topological $B_\com$ construction.

\begin{proof}
First observe that under the assumption on $\tau$ we have that $\eta_C: C\to \tau C$ is an isomorphism for any cyclic group $C$. Then by the description of the simplices of $\Wbar(\tau,K)$ given in Proposition \ref{pro:simplicesWbar} the map $\kappa$ factors through
$$
\kappa_\tau: \Sigma K \to \Wbar(\tau,K).
$$

The splitting is given by the following diagram
$$
\begin{tikzcd}
G\Wbar(\tau,K) \arrow{d} &  G (\Sigma K) \arrow[l,"G(\kappa_\tau)"'] & \\
G(\Wbar K) \arrow{r}{\sim} & K \arrow{r}& FK \arrow{ul}{\cong}
\end{tikzcd}
$$
Let us explain the maps. The weak equivalence is the counit of the adjunction between the loop group and bar construction.  In degree $n$ the simplicial group $FK$, known as \revC{Milnor's construction} \cite[page 285]{GJ99}, is the free group generated on  $K_n-\set{\ast}$.
The map $K\rightarrow FK$ sends an $n$-simplex to the corresponding generator of the free group.
The set of $n$-simplices of $G (\Sigma K)$ is given by $F(\Sigma K)_{n+1} /F(s_0(\Sigma K)_{n} )$ which can be identified with the $n$-simplices of $FK$ since $s_0(\Sigma K)_{n}$ maps onto the wedge summands of $(\Sigma K)_{n+1}$ other than $X_n$. This isomorphism is compatible with the simplicial structure. The last two maps are already described above. Starting from $K$ the composition of the five maps gives the identity. The splitting is natural with respect to $K$ since each construction is functorial in $K$.
\end{proof}

\Cor{\label{cor:split-surj}Under the assumption of Proposition \ref{pro:splitting} the natural map $\Wbar(\tau,K)\to \Wbar K$ induces a split surjection on homotopy groups.}

\section{Comparison to the topological version}  \label{sec:comptopcat}

{Throughout this section }$G$ denotes a topological group.
{In Section \ref{sec:varWbar} w}e introduced $\Wbar(\tau,K)$ and claimed that it is a simplicial analogue of $B(\tau,G)$. In this section we turn this claim into a theorem. We prove two homotopy equivalences
$$
B(\tau,|K|)\simeq |\Wbar(\tau,K)|,\;\;\;\; B(\tau,G) \simeq |\Wbar(\tau,SG)|
$$
where in the latter one $G$ is required to be a compact Lie group.

\subsection{Classifying spaces} The nerve construction  of a discrete group can be extended to the category of topological groups \cite{Seg68}. Given a topological group $G$ the set of $n$-simplices of $NG$ is the $n$-fold direct product $G^{\times n}=G\times \cdots \times G$, hence a topological space when $G$ has \revtwo{a} topology. The simplicial object $NG_\dt$ is a simplicial space, and its geometric realization is denoted by
$$
BG = |NG_\dt|
$$
which is called  the classifying space of $G$. We can think of the $n$-fold product as the space of group homomorphisms
$$
NG_n = \Hom(F^n,G).
$$
We use $\Hom(-,-)$, rather than $\catGrp(-,-)$, to emphasize the topology.

Given a simplicial group $K$, the geometric realization $|K|$ is a topological group, and we can consider $B|K|$. We would like to compare this space to the geometric realization of $\Wbar (K)$.

\Pro{\label{B_W}
There is a natural homotopy equivalence
$$
B|K| \rightarrow |\Wbar(K)|.
$$
}
\Proof{Taking the geometric realization of 
{the Cegarra--Remedios map $CR:dX\to TX$, which is a weak equivalence \cite{CR05}*{Theorem 1.1},}
we obtain a homotopy equivalence 
$$\revC{|CR|:} |dNK|\rightarrow |\Wbar K|$$ 
We can realize the bisimplicial set $NK$ in  different ways. All of these are homeomorphic to each other \cite{Qui73}. There is a sequence of  natural homeomorphisms
$$
\begin{aligned}
|dNK| &\cong | [p]\mapsto | [q]\mapsto  N(K_q)_p | | \\
&= |[p]\mapsto |K^p| |\\
& \cong  |[p]\mapsto |K|^p | \\
&= B|K|
\end{aligned}
$$
where we also used the fact that geometric realization preserves finite products \cite{Mil57}. \revC{Thus we obtain the desired map $B|K|\to |\Wbar(K)|$ as the composite $|CR|Q^{-1}$ where $Q:|dNK|\to B|K|$ is the homeomorphism described above.}
}

\subsection{$\tau$-version}\label{ss:B(tau,G)}
Given an endofunctor $\tau$  on the category of groups we can define
$$
B(\tau,G)=| [n]\mapsto \Hom(\tau^n, G ) |
$$
where $\tau^\dt = \tau F^\dt$ as usual. Note that this definition works for an arbitrary cosimplicial group not necessarily coming from an endofunctor. For the rest of the section we assume that $\tau$ is of quotient type.
In this case  the homomorphism space $\Hom(\tau^n,G)$ is topologized as a subspace of $G^{\times n}$. The natural transformation $\eta$ induces an inclusion $B(\tau,G)\subset BG$. We denote by $E(\tau,G)\to B(\tau,G)$ the pull-back of the universal bundle $EG\to BG$.

We would like to prove an analogue of Proposition \ref{B_W}. To prepare we need a preliminary result. Observe that the set of group homomorphisms $\Hom(\tau F^n,K_m)$ can be assembled into a simplicial set by using the simplicial structure of $K$.

\Lem{\label{lem:real-T}
There is a natural homeomorphism
$$
|\Hom(\tau F^n,K_\dt)| \rightarrow \Hom(\tau F^n , |K|).
$$
}
\begin{proof}    \revC{
We will use some of the basic properties of the geometric realization proved in \cite[Chapter III]{May67}.
A point in the geometric realization $|X|$ of a simplicial set $X$ is given by an equivalence class $[u_m,x_m]$ where $x_m$ is a non-degenerate $m$-simplex of $X$ and $u_m$ is a point in the interior of $\Delta^m$. The natural map $\phi:|X\times X|\to |X|\times |X|$ induced by applying $|\,|$ to the projection maps onto each factor of $X\times X$ is a homeomorphism. Let $\phi^{-1}:|X|\times |X|\to |X\times X|$ denote the inverse of this map. This map can be described explicitly (as in the proof of \cite[Theorem 14.3]{May67}), but we just need to know that a point $([u_{m_1},x_{m_1}],[u_{m_2},x_{m_2}])$ in the product is sent to the point $[u_m,(x^{(1)}_m,x^{(2)}_m)]$ where $u_{m_i},\;i=1,2,$ can be obtained from $u_m$ by applying a sequence of codegeneracy maps and $x^{(i)}_{m},\;i=1,2,$ are given by applying the corresponding sequence of degeneracy maps to $x_{m_i}$. Note that the multiplication map of $|K|$ is given by $|K|\times |K|\xrightarrow{\phi^{-1}} |K\times K| \xrightarrow{|\mu|} |K|$.
Let $\phi_n$ denote the map $|K^{\times n}|\to |K|^{\times n}$ obtained by applying $\phi$ multiple times: $|K^{\times n}| \to |K|\times |K^{\times {n-1}}| \to \cdots \to |K|^{\times n} $. Let us define the following maps
$$
\iota_1: |\Hom(\tau F^n,K_\dt)| \to |K^{\times n}|
$$
defined by $\iota_1([u_m,\tau F^n \xrightarrow{f} K_m])=[u_m,(f(e_1),\cdots,f(e_n))]$, and
$$
\iota_2: \Hom(\tau F^n,|K|) \to |K|^{\times n}
$$
defined by $\iota_2(\tau F^n\xrightarrow{g} |K| )=(g(e_1),\cdots,g(e_n))$. Both of these maps are embeddings. This is clear for $\iota_1$. For $\iota_2$ it follows from the fact that the geometric realization functor commutes with finite limits, in particular equalizers.

We will write $R_n$ for the kernel of $\eta_{F^n}: F^n\to\tau  F^n$. Each element of $R_n$ determines a relation $r(e_1,\cdots,e_n)$. A group homomorphism $F^n\to H$ factors through $\tau F^n\to H$ if and only if $r(f(e_1),\cdots,f(e_n))=1$ in $H$ for all $r\in R_n$. Now, consider the composite $\phi_n \iota_1$ which maps $[u_m,\tau F^n \xrightarrow{f} K_m]$ to the tuple $([u_m,f(e_1) ],\cdots,[u_m,f(e_n)])$. This tuple regarded as a group homomorphism $F^n\to |K|^{\times n}$ factors through $\tau F^n$. This follows from the fact that the elements $x_i=[u_m,f(e_i)]$ satisfy the relations $r(x_1,\cdots,x_n)=1$. Note that here we are using  the fact that $\phi^{-1}([u_m,k_m],[u_m,k'_m])$ is simply $[u_m,(k_m,k'_m)]$.
Therefore we obtain a commutative diagram
}
$$
\begin{tikzcd}
{|\Hom(\tau F^n,K_\dt)|} \arrow[d,hook,"\iota_1"] \arrow{r} & {\Hom(\tau F^n , |K|)} \arrow[d,hook,"\iota_2"]  \\
{|K^{\times n}|} \arrow[r,"\phi_n"]  & {|K|^{\times n}}
\end{tikzcd}
$$
\revC{We will show that the top map is a homeomorphism. Injectivity of the map is clear from the diagram.  To see surjectivity let $ f:F^n\rightarrow |K|$ be a homomorphism that factors through $\tau F^n$. Each element $f(e_i)$ is given by an  equivalence class $[u_{l_i},k_{l_i}]$   where $k_{l_i}$ is an $l_i$-simplex of $K$ and $u_{l_i}$ is a point in the interior of $\Delta^{l_i}$
Let $[u_m, k_{m}^{(1)},\cdots, k_{m}^{(n)} ]$ denote the element which maps to $([u_{l_1},k_{l_1}],\cdots, [u_{l_n},k_{l_n}]) $ under the homeomorphism $\phi_n$.
Here $k_{m}^{(i)}$ belongs to $K_m$ and $u_m$ belongs to the interior of $\Delta^m$.
Then for each $r\in R_n$ the relation $r(k_{m}^{(1)},\cdots, k_{m}^{(n)})=1$ holds in $K_m$ since $r(f(e_1),\cdots,f(e_n))=1$ in $|K|$.
Sending $e_i$ to $k_m^{(i)}$ defines a homomorphism $F^n\rightarrow K_m$ which factors through $f':\tau F^n\rightarrow K_m$. The element $[u_m,f']$ in the geometric realization ${|\Hom(\tau F^n,K_\dt)|}$ maps to $f$ under $\phi_n$. This proves the surjectivity. As a result  the homeomorphism $\phi_n$ restricts to a bijection between the subspaces, and consequently gives the desired homeomorphism.  }
\end{proof}

\Thm{\label{thm:B-Wbar}
There is a natural homotopy equivalence
$$
B(\tau,|K|) \rightarrow |\Wbar(\tau,K)|.
$$
}
\Proof{
Recall that we have a weak equivalence $dY\rightarrow TY$ for any bisimplicial set $Y$. We apply this to the bisimplicial set $N(\tau,K)$, and use the identification $\Wbar(\tau,K)\cong TN(\tau,K)$. This gives a homotopy equivalence $$
|dN(\tau,K)| \rightarrow |\Wbar(\tau,K)|
$$
after realization. It remains to identify $|dN(\tau,K)|$ with $B(\tau,|K|)$. Using Lemma \ref{lem:real-T} and the homeomorphism between different ways of realizing a bisimplicial \revC{set} we obtain
$$
\begin{aligned}
|dN(\tau,K)| &\cong |[p]\mapsto | [q]\mapsto  \Hom(\tau F^p,K_q) | | \\
& \cong
 |[p]\mapsto \Hom(\tau F^p , |K|) | \\
 &= B(\tau,|K|).
 \end{aligned}
$$
}

\subsection{Compact Lie groups}
In Theorem \ref{thm:B-Wbar} we started from a simplicial group $K$ and compared the simplicial and topological constructions. Conversely we can do such a comparison for a topological group $G$. However, for such a comparison to work we need to restrict our attention to a nice class of groups such as compact Lie groups. \revC{We introduce the class of $\tau$-good groups in Definition \ref{def:GoodGroup}.}

\Rem{\label{rem:fat-real}{\rm
The main reason for this restriction is that the geometric realization functor does not respect homotopy equivalences in general. A better behaving realization functor is the {\it fat realization} which is obtained by forgetting the degeneracies when gluing the simplices {\cite[Appendix A]{segal1974categories}}. For a simplicial space $X$ its fat realization is denoted by $\fat{X}$. If the simplicial space is {\it good}, i.e. all degeneracy maps $s_i:X_{n-1}\to X_n$ are closed cofibrations in the sense of Hurewicz, then the natural map $\fat{X}\to |X|$ is a homotopy equivalence. Any simplicial set is a good simplicial space, or more generally the simplicial space $[n]\mapsto |S(X_n)|$ is good where $X$ is an arbitrary simplicial space.

}}

\begin{rem}\label{rem:ComLie-Good}
If $G$ is a compact Lie group then $[n]\mapsto \Hom(\tau^n,G)$ is good. Thus up to homotopy we can replace geometric realization by the fat realization in the construction of $B(\tau,G)$, see \cite{OW19}.  This property will be crucial in our consideration.
\end{rem}

We need a  version of Lemma \ref{lem:real-T} for the singular functor. \revC{
 Let $H$ be a discrete group and $G$ a topological group. Recall that the mapping space $\Map(H,G)$ is the set of maps $H\to G$ equipped with the compact-open topology. We equip the subset $\catGrp(H,G)\subseteq\Map(H,G)$ with the subspace topology. In the case of $H=\tau F^n$ this is consistent with the topology induced from $G^{\times n}$
}

\Lem{\label{lem:sing}
Let $H$ be a discrete group, and $G$ a topological group. Then there is a natural isomorphism of simplicial sets
$$
S(\Hom(H,G)) \to \Hom(H,S(G)_\dt).
$$
}
\begin{proof}
\revC{
It is enough to show
that the pair of mutually inverse natural isomorphisms
\begin{center}

\begin{tikzpicture}[xscale=3,yscale=1.5]
\node (C) at (0,0) {$\catTop(\Delta^n,\Map(H,G))$};
\node (D) at (1,0) {$\cong$};
\node (E) at (2,0) {$\catSet(H,\catTop(\Delta^n,G))$};
\path[->,font=\scriptsize,>=angle 90]
(C) edge [bend left=60] node [above] {$f\mapsto(h\mapsto\ev_h\circ f)$} (E)
(E) edge [bend left=60] node [below] {$(x\mapsto(\,h\mapsto w(h)(x)\,  ))\mapsfrom w$} (C);
\end{tikzpicture}

\end{center}
restricts to a pair of mutually inverse natural isomorphisms
$$
\catTop(\Delta^n,\Hom(H,G))\cong
\catGrp(H,\catTop(\Delta^n,G)).
$$
Take $f\in \catTop(\Delta^n,\Hom(H,G))$, $h,h'\in H$ and $x\in\Delta^n$. Then we have
\begin{align*}
 (\ev_h\circ f)\cdot(\ev_{h'}\circ f)(x)&=(f(x)(h))\cdot(f(x)(h'))\\
 &=f(x)(hh')\\
 &=(\ev_{hh'}\circ f)(x)
\end{align*}
which shows $(h\mapsto\ev_h\circ f)\in\catGrp(H,\catTop(\Delta^n,G))$.

Take $w\in \catGrp(H,\catTop(\Delta^n,G))$. Then we have
\begin{align*}
 w(h)(x)\cdot w(h')(x)&=(w(h)\cdot w(h'))(x)\\
 &=w(hh')(x)
\end{align*}
which shows $(x\mapsto(\,h\mapsto w(h)(x)\,))\in\catTop(\Delta^n,\Hom(H,G))$.
}
\end{proof}

\revC{
\Def{\label{def:GoodGroup}
A topological group $G$ is said to be $\tau$-good if each $B(\tau,G)_n$ is a good simplicial space consisting of CW-complexes (simplicial structure maps are {cellular maps}).
}
}

\Cor{\label{cor:simp-Lie}
\revC{
Let $G$ be a topological group which is $\tau$-good, and let $K$ be a simplicial group.
\begin{enumerate}
\item We have a chain of homotopy equivalences
\begin{equation}\label{eq:G-|SG|}
|\Wbar (\tau,SG)|\simeq B (\tau, |SG|) \simeq B(\tau,G)
\end{equation}
where the first one is the one in Theorem \ref{thm:B-Wbar} and the second one is induced by the counit $\epsilon_G: |S G|\to G$.

\item The unit $\eta_K:K\to S|K|$ induces a weak equivalence
\begin{equation}\label{eq:K-S|K|}
\Wbar(\tau,K)\to \Wbar(\tau,S|K|).
\end{equation}
\end{enumerate}
}
}
\begin{proof}
\revC{
It is enough to show that the map $B(\tau,|SG|)\to B(\tau,G)$ induced by the counit $\epsilon_G:|SG|\to G$ is a homotopy equivalence.} By Remark \ref{rem:fat-real} we can use fat realization \revC{for good simplicial spaces}. Lemma \ref{lem:real-T}, Lemma \ref{lem:sing}, and the  \revC{weak} equivalence \revC{$\epsilon_G:|S G| \to G$}  gives us the diagram
 $$
\begin{tikzcd}
{|S\Hom(\tau F^n,G)|} \arrow[r,"\sim"] \arrow[d,"\cong"] & {\Hom(\tau F^n, G)} \\
{|\Hom(\tau F^n, (SG)_\dt)|} \arrow[r,"\cong"] & {\Hom(\tau F^n,|S G|)} \arrow{u}
\end{tikzcd}
$$
\revC{The top map is a homotopy equivalence since $\Hom(\tau F^n,G)$ is a CW-complex.
}
Thus we obtain a level-wise homotopy equivalence, which after  realization, gives \revC{the} homotopy equivalence \revC{in (\ref{eq:G-|SG|})}.

\revC{
To prove part (2) we will show that $|\Wbar(\tau,K)|\to |\Wbar(\tau,S|K|)|$ is a homotopy equivalence.  By Theorem \ref{thm:B-Wbar} this amounts to showing that $B(\tau,|K|)\to B(\tau,|S|K||)$ induced by $|\eta_K|$ is a homotopy equivalence. First observe that $B(\tau,|K|)$ is a good simplicial space for any simplicial group $K$: Each degeneracy map $s_i:\Hom(\tau F^n, |K|) \to \Hom(\tau, F^{n+1}, |K|)$ is a closed cofibration since by Lemma \ref{lem:real-T} $|s_i|:|\Hom(\tau F^n, K)| \to |\Hom(\tau F^{n+1}, K)|$ is an inclusion of a subcomplex of a CW-complex. Therefore  similar to part (1) we can argue level-wise. Now, consider the map
$
\Hom(\tau F^n,|K|) \to \Hom(\tau F^n,|S|K||)
$
induced by $|\eta_K|$. By Lemma \ref{lem:real-T} and Lemma \ref{lem:sing} it is a direct calculation to check that, up to homeomorphism, this map coincides with the weak equivalence
$$|\eta_{\Hom(\tau F^n,K)}|:|\Hom(\tau F^n,K)| \to |S|\Hom(\tau F^n,K)||.$$ This is a homotopy equivalence since the spaces involved are CW-complexes. After geometric realization we obtain the desired homotopy equivalence.
}
\end{proof}

\bibliography{bib}{}
\bibliographystyle{alpha}

\appendix

\end{document}